\documentclass[11pt,letterpaper]{amsart}

\numberwithin{equation}{section}
\theoremstyle{plain}
\newtheorem{theorem}{Theorem}[section]
\newtheorem{prop}[theorem]{Proposition}

\newtheorem{lemma}[theorem]{Lemma}

\newtheorem{corollary}[theorem]{Corollary}

\theoremstyle{definition}

\newtheorem{remark}[theorem]{Remark}

\newcommand{\Rmnum}[1]{\expandafter\@slowromancap\romannumeral #1@}

\newcommand{\mr}{\mathbb{R}}
\newcommand{\ud}{\mathrm{d}}
\newcommand{\ms}{\mathbb{S}}

\allowdisplaybreaks

\address{Mingxiang Li, Department  of Mathematics \& Institue of Mathematical Sciences, Chinese University of Hong Kong}
\email{mingxiangli@cuhk.edu.hk}

\begin{document}
	\title{The total Q-curvature, volume entropy and polynomial growth polyharmonic functions (II)}
	\author{Mingxiang Li}
\date{}
\maketitle
	\begin{abstract}
		This is a continuation of our previous work (Advances in Mathematics 450 (2024), Paper No. 109768). In this paper, we  characterize complete metrics with finite total Q-curvature as normal metrics for all dimensional cases. Secondly, we introduce another volume entropy to provide geometric information regarding complete non-normal metrics with finite total Q-curvature. In particular, we show that if the scalar curvature is bounded from below, the volume growth of such complete metrics is controlled.
	\end{abstract}

	\section{Introduction}
	
Given a  smooth and complete  conformal metric $g=e^{2u}|dx|^2$ on $\mr^n$ where the integer $n\geq 2$, the Q-curvature  satisfies 
 the conformally invariant equation:
	\begin{equation}\label{equ:conformal eqution}
		(-\Delta)^{\frac{n}{2}}u(x)=Q(x)e^{nu(x)}, \quad x\in\mr^n.
	\end{equation}
When $n$ is odd, the operator $(-\Delta)^{\frac{n}{2}}$ is  non-local, and its precise definition  will be discussed in Section \ref{sec:normal metrics}.
We say that  the metric $g=e^{2u}|dx|^2$ on $\mr^n$ is equipped with   finite total Q-curvature if  the Q-curvature is well-defined and  $Qe^{nu}\in L^1(\mr^n).$
For more information about Q-curvature, we recommend interested readers to refer to the introduction of \cite{Li 23 Q-curvature}. More properties about conformally invariant  equation \eqref{equ:conformal eqution}  with constant $Q$ can be found in \cite{JMMX}, \cite{Lin}, \cite{Mar MZ}, \cite{WX},  and the references therein.

The famous Cohn-Vossen inequality (See \cite{Cohn}, \cite{Hu}) shows that  for $n=2$ in  the equation \eqref{equ:conformal eqution}, if the metric $g=e^{2u}|dx|^2$ is complete with finite total Q-curvature, then the following  holds
\begin{equation}\label{CV ineq}
	\int_{\mr^2}Qe^{2u}\ud x\leq 2\pi.
\end{equation}
For $n=2$, the Q-curvature precisely coincides with the Gaussian curvature of the conformal metric. In Huber's work \cite{Hu}, a broader class of complete non-compact surfaces is examined.
During  dealing with  such problems, an important concept called the normal metric was introduced by Finn \cite{Finn} for $n=2$. The analogous definition for higher dimensional cases  is considered by Chang, Qing, and Yang \cite{CQY}.  Precisely, the conformal metric $g=e^{2u}|dx|^2$  with finite total Q-curvature is said to be  normal if  the solution $u(x)$ to  the equation \eqref{equ:conformal eqution} satisfies the integral equation
\begin{equation}\label{integral equation}
	u(x)=\frac{2}{(n-1)!|\mathbb{S}^n|}\int_{\mr^n}\log\frac{|y|}{|x-y|}Q(y)e^{nu(y)}\ud y+C
\end{equation}
where $|\mathbb{S}^n|$ denotes the volume of standard sphere $\mathbb{S}^n$ and $C$ is a constant which may be different from line to line throughout this paper.
Briefly, normal solutions to \eqref{equ:conformal eqution} can be represented by using the Green's function of the operator $(-\Delta)^{\frac{n}{2}}$ via a  potential. Since the kernel of the operator $(-\Delta)^{\frac{n}{2}}$ on $\mr^n$ includes  more than constants, not all solutions to \eqref{equ:conformal eqution} are normal. Hence, it is necessary to add some conditions to ensure  the solutions to be normal. To serve this aim, Chang, Qing, and Yang \cite{CQY} established a sufficient condition for normal metrics by requiring that the scalar curvature  $R_g$ is non-negative near infinity. This is a very comprehensive condition since it is geometric and reveals that the scalar curvature plays an important role in studying higher-dimensional problems related to Q-curvature. Recently, the author \cite{Li 24 revisited} followed their steps to further investigate the relationship between Q-curvature and scalar curvature. Moreover,  in \cite{CQY}, they generalized the Cohn-Vossen inequality \eqref{CV ineq} to the higher dimensional case and showed that  if the complete metric $g=e^{2u}|dx|^2$ on $\mr^4$ with finite total Q-curvature is normal, the integral of Q-curvature is bounded from above.  Subsequently, a lot of works are established in this direction including \cite{BN1}, \cite{BN2}, \cite{CQY2}, \cite{Fa}, \cite{Li 24 revisited}, \cite{LuWang}, \cite{NX}, \cite{Wang} and the references therein.

Is there another geometric condition to ensure  normal metrics?  Furthermore, does  a necessary and sufficient condition for normal metrics exist?  In \cite{Li 23 Q-curvature}, the author  introduced  a volume entropy to achieve this goal.
The volume entropy is defined as follows 
$$\tau(g):=\lim_{R\to\infty}\sup\frac{\log V_g(B_R(0))}{\log|B_R(0)|}$$
where $B_R(0)$ denotes the Euclidean ball with radius $R$ centered at the origin and   $V_g(B_R(0))$, $|B_R(0)|$ represent the volume of $B_R(0)$ related to the metric $g$ and Euclidean metric respectively.   Theorem 1.1 in \cite{Li 23 Q-curvature}  showed that the complete metric $g$ is normal if and only if $\tau(g)$ is finite for even integer $n$. Firstly, we aim to generalize such result to all dimensions.
\begin{theorem}\label{thm: tau(g)}
	Consider a  smooth and complete   metric $g=e^{2u}|dx|^2$ on $\mr^n$  with finite total Q-curvature where  the integer $n\geq 2$. Then the  metric $g$ is normal if and only if $\tau(g)$ is finite.  Moreover, if the volume entropy $\tau(g)$ is finite, one has
	$$\tau(g)=1-\frac{2}{(n-1)!|\mathbb{S}^n|}\int_{\mr^n}Qe^{nu}\ud x.$$
\end{theorem}

With the help of the above theorem, we believe that we have a relatively complete description of normal metrics. Now, it is natural to address non-normal metrics.  For $n=2$, Huber \cite{Hu}(See also \cite{Finn}) showed that all complete metrics with finite total Q-curvature are normal. However, for $n\geq 3$, the situation becomes very different. In fact,  it is not hard  to construct  non-normal  complete metrics with finite total Q-curvature on $\mr^n$  when $n\geq 3$. 
 For example, choose the smooth function
 \begin{equation}\label{non-normal example}
 	u(x)=-\beta\log(|x|^2+1)+|x|^2
 \end{equation}
 where $\beta \in \mr$. For $n\geq 3$, it is not hard to check that the metric $g=e^{2u}|dx|^2$ is complete, $(-\Delta)^{\frac{n}{2}}u\in L^1(\mr^n)$ and the integral of Q-curvature satisfies
$$\int_{\mr^n}Qe^{nu}\ud x=(n-1)!|\ms^n|\beta.$$
This example also indicates that if we permit the complete metrics to be non-normal, we will lose control over the integral of the Q-curvature. It is natural to inquire whether we can achieve other  geometric description for non-normal complete metrics with finite total Q-curvature.

As  mentioned earlier, Chang, Qing, and Yang \cite{CQY} showed that the metric 
$g$ is normal if the scalar curvature  $R_g\geq 0$ near infinity. For non-normal metrics as given in \eqref{non-normal example}, the scalar curvature $R_g$ has a uniform lower bound via  a direct computation
$$R_g=2(n-1)e^{-2u}(-\Delta u-\frac{n-2}{2}|\nabla u|^2)\geq -Ce^{-2|x|^2}(|x|^2+1)^{1+2\beta}\geq -C.$$
  Therefore, it is natural to ask: if we replace the condition 
$R_g\geq 0$ near infinity with that  $R_g$
has a uniform lower bound, can we obtain some interesting geometric properties of non-normal metrics?  Theorem \ref{thm: tau(g)} shows that volume growth is a significant geometric property related to the structure of metrics. Therefore, it is helpful  to consider volume growth for non-normal metrics. However, for such metrics, $\tau(g)$ must be infinite. Consequently, we need an alternative  volume entropy to address non-normal metrics. So we  
introduce the following  volume entropy defined as 
$$h(g):=\inf\{s|\lim_{R\to\infty}\sup \frac{\log V_g(B_R(0))}{R^s}<+\infty,\quad s\geq 0 \}$$
with the  convention $\inf\emptyset=+\infty$. With  help of this new volume entropy $h(g)$, we  provide some volume growth control for non-normal complete metrics under the assumption that the scalar curvature is bounded from below.

\begin{theorem}\label{thm: h(g) leq n-2}
	Consider a  smooth and complete metric $g=e^{2u}|dx|^2$ on $\mr^n$  with finite total Q-curvature where the integer $n\geq 3$. Suppose that the scalar curvature $R_g\geq -C$ for some constant $C>0$. Then  the volume entropy $h(g)$ is an even integer  satisfying 
		$$0\leq h(g)\leq n-1.$$
\end{theorem}
\begin{remark}
	Consider the  complete metrics $g_k=e^{2u_k}|dx|^2$ where  $u_k(x)=|x|^{2k}$ for  the integer $k$ satisfying $1\leq k\leq \frac{n-1}{2}$.  A direct computation yields that 
	$Q_{g_k}\equiv 0$, $R_{g_k}\geq -C$ and the volume entropy $h(g_k)=2k.$ Thus each possible volume entropy can be achieved.
\end{remark}

In fact, we are able to give a precise decomposition for such complete metrics with scalar curvature bounded from below.

\begin{theorem}\label{thm: normal decomposition}
	Consider a   smooth and complete  metric $g=e^{2u}|dx|^2$ on $\mr^n$  with finite total Q-curvature where the integer $n\geq 3$. Suppose that the scalar curvature $R_g\geq -C$ for some constant $C>0$. Then $u(x)$ has the following decomposition
	$$u(x)=\frac{2}{(n-1)!|\mathbb{S}^n|}\int_{\mr^n}\log\frac{|y|}{|x-y|}Q(y)e^{nu(y)}\ud y+P(x)$$
	where $P(x)$ is a polynomial with even degree  $\deg (P)\leq n-1$ and 
	$$P(x)\geq -C\log(|x|+2)$$
	for some constant $C>0.$
\end{theorem}

If we introduce additional control over the scalar curvature, we can achieve more precise control over the volume growth, leveraging the results established in \cite{Li 24 revisited} and Theorem \ref{thm: tau(g)}.
\begin{theorem}\label{thm: tua(g) leq 1}
	Consider a  smooth and complete metric $g=e^{2u}|dx|^2$ on $\mr^n$  with finite total Q-curvature where the integer $n\geq 3$. Suppose that the scalar curvature $R_g\geq 0$ near infinity. Then  the volume entropy  
	$$ \tau(g)\leq 1.$$
\end{theorem}
\begin{remark}
Clearly, one can verify that the standard metric  $|dx|^2$ meets all the  conditions needed, and a direct computation shows that $\tau(g)$  equals one. Therefore, the upper bound of $\tau(g)$  is sharp. However, it does not imply rigidity unless additional constraints on the sign of the Q-curvature are imposed (see Corollary 1.3 in \cite{Li 23 Q-curvature}).
\end{remark}
\begin{theorem}\label{thm: tau(g) =0}
	Consider a  smooth and complete metric $g=e^{2u}|dx|^2$ on $\mr^n$  with finite total Q-curvature where the integer $n\geq 3$. Suppose that the scalar curvature $R_g\geq C$  near infinity for some constant $C>0$. Then  the volume entropy 
	$$\tau(g)= 0.$$
\end{theorem}

This  paper is organized as follows. In Section \ref{sec:normal metrics}, we generalize the result in \cite{Li 23 Q-curvature} to all dimensional cases and prove Theorem \ref{thm: tau(g)}. In Section \ref{section:non-normal metrics}, we address non-normal complete metrics and prove Theorem \ref{thm: h(g) leq n-2} and Theorem \ref{thm: normal decomposition}. Finally, in Section \ref{sec: finite tau(g)}, we prove Theorem \ref{thm: tua(g) leq 1} and Theorem \ref{thm: tau(g) =0}.

\hspace{3em}

{\bf Acknowledgment.} Part of this work was completed during the author's visits to Great Bay University and Hangzhou Dianzi University. The author is grateful  for their  hospitality. Additionally, the author would like to thank Dr. Zhehui Wang for his helpful discussions.

\section{Normal complete metrics}\label{sec:normal metrics}

As mentioned in the introduction, when the integer $n$  is odd, the  operator $(-\Delta)^{\frac{n}{2}}$ is  non-local and interpreted as
$$(-\Delta)^{\frac{n}{2}}=(-\Delta)^{\frac{1}{2}}\circ(-\Delta)^{\frac{n-1}{2}}.$$
 Given a smooth function $f(x)$ belong to $L_{1/2}(\mr^n)$ which is defined by
$$L_{1/2}(\mr^n):=\{\varphi\in L^1_{loc}(\mr^n)| \int_{\mr^n}\frac{|\varphi(x)|}{1+|x|^{n+1}}\ud x<+\infty\}.$$
Then  $(-\Delta)^{\frac{1}{2}}f$  can be written as
$$(-\Delta)^{\frac{1}{2}}f(x):=C(n)P.V.\int_{\mr^n}\frac{f(x)-f(y)}{|x-y|^{n+1}}\ud y$$
where the right-hand side is defined in the sense of the principal value and $C(n)$ is a constant depending on $n$.  It can also be defined as pseudo-differential operator via the Fourier transform
\begin{equation}\label{Fourier transform}
	\widehat{(-\Delta)^{\frac{1}{2}}f}(\xi)=|\xi|\hat{f}(\xi).
\end{equation}
With help of such an  equivalent definition  \eqref{Fourier transform}, the following identity holds
$$\widehat{(-\Delta)^{\frac{1}{2}}\circ(-\Delta)^{\frac{1}{2}}f}(\xi)=|\xi|\widehat{(-\Delta)^{\frac{1}{2}}f}(\xi)=|\xi|^2\hat{f}(\xi)=\widehat{(-\Delta)f}(\xi).$$
Thus, a  property which will be used later is obtained as follows
\begin{equation}\label{Delta 1/2 circ 1/2=1}
	(-\Delta)^{\frac{1}{2}}\circ(-\Delta)^{\frac{1}{2}}f=(-\Delta) f.
\end{equation}
More equivalent definitions of the fractional  operator $(-\Delta)^{s}$,  where $s\in (0,1)$,  can be found in the introduction of \cite{CS}.  More details about fractional Laplacian operator  from an analytic point of view can be found in \cite{CS}, \cite{Silvestre}, and from a geometric point of view  in \cite{CC},  \cite{CG}, \cite{FG},  \cite{GQ} and \cite{GZ}. We also recommend the interested reader to refer to a nice survey \cite{Gon} by Gonz\'alez.

When  the integer $n$ is odd,  to make the equation \eqref{equ:conformal eqution}  well-defined, we need to  assume that  $(-\Delta)^{\frac{n-1}{2}}u$ belongs to $L_{1/2}(\mr^n)$  in addition,  such that  $(-\Delta)^{\frac{n}{2}}u=(-\Delta)^{\frac{1}{2}}\circ(-\Delta)^{\frac{n-1}{2}}u$ is well-defined.  For even integer $n$, it is well known that the Green's function  $G(x,y)$ of the operator $(-\Delta)^{\frac{n}{2}}$ on $\mr^n$ is  
$$G(x,y)=\frac{2}{(n-1)!|\mathbb S^n|}\log\frac{1}{|x-y|}.$$  As for odd integer $n$, the Green's function of the non-local operator  is same as the even cases(See Proposition 2.1 in \cite{NX}). 
Generally, given  a smooth function $f(x)\in L^1(\mr^n)$, we set the logarithmic potential defined as
\begin{equation}\label{mathcal L}
	\mathcal{L}(f)(x):=\frac{2}{(n-1)!|\mathbb{S}^n|}\int_{\mr^n}\log\frac{|y|}{|x-y|}f(y)\ud y
\end{equation}
which satisfies the equation
\begin{equation}\label{delta Lf}
	(-\Delta)^{\frac{n}{2}}\mathcal{L}(f)(x)=f(x).
\end{equation}
Set the notation $\alpha$ as
$$\alpha:=\frac{2}{(n-1)!|\mathbb{S}^n|}\int_{\mr^n}f(y)\ud y.$$

For the reader's convenience, we repeat the statements of some lemmas which have been established in \cite{Li 23 Q-curvature}.

\begin{lemma} \label{lem: L(f)}(Lemma 2.3 in \cite{Li 23 Q-curvature})
	For $|x|\gg1$, there holds
	\begin{equation}\label{Lf=-aplha log x+}
		\mathcal{L}(f)(x)=(-\alpha+o(1))\log|x|+\frac{2}{(n-1)!|\mathbb{S}^n|}\int_{B_1(x)}\log\frac{1}{|x-y|}f(y)\ud y
	\end{equation}
	where  $o(1)\to 0$ as $|x|\to\infty$.
\end{lemma}

\begin{lemma}\label{lem: B_1 L(f)}(Lemma 2.4 in \cite{Li 23 Q-curvature})
	For any $r_0>0$ fixed, there holds
	\begin{equation}\label{B_1(x)Lf}
		\frac{1}{|B_{r_0}(x)|}\int_{B_{r_0}(x)}\mathcal{L}(f)(y)\ud y=(-\alpha+o(1))\log|x|
	\end{equation}
where  $o(1)\to 0$ as $|x|\to\infty$.
\end{lemma}

\begin{lemma}\label{lem: B_r_1|x| L(f)}(Lemma 2.5 in \cite{Li 23 Q-curvature})
	For any $0<r_1<1$ fixed, there holds
	\begin{equation}\label{B_r_1|x|(x)Lf}
		\frac{1}{|B_{r_1|x|}(x)|}\int_{B_{r_1|x|}(x)}\mathcal{L}(f)(y)\ud y=(-\alpha+o(1))\log|x|
	\end{equation}
where  $o(1)\to 0$ as $|x|\to\infty$.
\end{lemma}
\begin{lemma}\label{lem: B_|x|^p L(f)}(Lemma 2.6 in \cite{Li 23 Q-curvature})
	For any $r_2>0$ fixed and $|x|\gg1$, there holds
	\begin{equation}\label{B_r_2(x)Lf}
		\frac{1}{|B_{|x|^{-r_2}}(x)|}\int_{B_{|x|^{-r_2}}(x)}\mathcal{L}(f)(y)\ud y=(-\alpha+o(1))\log|x|
	\end{equation}
where  $o(1)\to 0$ as $|x|\to\infty$.
\end{lemma}

\begin{lemma}\label{lem: e^nLf on B_R+1 B_R-1}(Lemma 2.8 in \cite{Li 23 Q-curvature})
	For $R\gg1$, there holds
	$$\int_{B_{R+1}(0)\backslash B_{R-1}(0)}e^{n\mathcal{L}(f)}\ud y=R^{n-1-n\alpha+o(1)}$$
	where $o(1)\to 0$ as $R\to\infty.$
\end{lemma}

\begin{lemma}\label{lem: B_R(0)Lf}(Lemma 2.11 in \cite{Li 23 Q-curvature})
	For $R\gg1$, there holds
	$$\int_{B_R(0)}|\mathcal{L}(f)|\ud x=O((\log R)\cdot R^n).$$
\end{lemma}
	\begin{lemma}\label{lem: polynomial growth is polynomial}(See \cite{Armitage 01} or Lemma 2.12 in \cite{Li 23 Q-curvature})
	Foy any integer $m\geq 1,k\geq 0$, suppose that  $$\Delta^m\varphi(x)=0\quad  \mathrm{and}\quad  
	\int_{B_R(0)}\varphi^+\ud x=o(R^{k+n+1}).$$
	Then $\varphi(x)$ is a polynomial with 
	$\deg\varphi\leq \max\{2m-2,k\}.$
\end{lemma}

The following lemma has been  established by the estimate (3.6) in \cite{Li 23 Q-curvature}. For the reader's  convenience, we repeat its proof and establish it as  the following lemma.

\begin{lemma}\label{lem: line e^u}
	Consider a curve $\gamma(t)=tx_0$ where $x_0\in \mr^n$ satisfying $|x_0|=1$ and $t\geq 0$.	For  $i\gg1$,  there holds
	$$\int^{i+1}_{i}e^{\mathcal{L}(f)(\gamma(t))}\ud t \leq i^{-\alpha+o(1)}$$
	where $o(1)\to 0$ as $i\to\infty.$
\end{lemma}
\begin{proof}
	Firstly, notice that
	$$|\int_{B_1(z)\backslash B_{1/4}(z)}\log\frac{1}{|z-y|}f(y)\ud y|\leq \log 4\int_{B_1(z)\backslash B_{1/4}(z)}|f(y)|\ud y\leq C.$$
	Making use of  Lemma \ref{lem: L(f)}, for $|z|\gg1$, 
	there holds
	$$\mathcal{L}(f)(z)= (-\alpha+o(1))\log|z|+\frac{2}{(n-1)!|\mathbb{S}^n|}\int_{B_{1/4}(z)}\log\frac{1}{|z-y|}f(y)\ud y.$$
	Then 	we can easily obtain that 
	$$\mathcal{L}(f)(z)\leq (-\alpha+o(1))\log|z|+\frac{2}{(n-1)!|\mathbb{S}^n|}\int_{B_{1/4}(z)}\log\frac{1}{|z-y|}f^+(y)\ud y.$$
	Moreover, for $z\in B_{1/4}(x)$ and $|x|\gg1$, there holds
	\begin{equation}\label{equ: u leq f^+}
		\mathcal{L}(f)(z)\leq (-\alpha+o(1))\log|x|+\frac{2}{(n-1)!|\mathbb{S}^n|}\int_{B_{1/2}(x)}\log\frac{1}{|z-y|}f^+(y)\ud y.
	\end{equation}
	If $f^+=0$ a.e. on $B_{1/2}(x)$, for   any $z\in B_{1/4}(x)$ ,
	there holds 
	$$\mathcal{L}(f)(z)\leq (-\alpha+o(1))\log|x|.$$
	Then result  follows in this case.
	Otherwise, since $f\in L^1(\mr^n)$, there exists $R_1>0$ such that for any $|x|>R_1$, 
	$$\int_{B_{1/2}(x)}f^+(y)\ud y\leq \frac{(n-1)!|\mathbb{S}^n|}{4}.$$ 
	The estimate \eqref{equ: u leq f^+} and Jensen's inequality yield that for $|x|>R_1$ and $z\in B_{1/4}(x)$
	\begin{align*}
		e^{\mathcal{L}(f)(z)}\leq &|x|^{-\alpha+o(1)}\exp\left(\frac{2}{(n-1)!|\mathbb{S}^n|}\int_{B_{1/2}(x)}\log\frac{1}{|z-y|}f^+(y)\ud y\right)\\
		\leq &|x|^{-\alpha+o(1)}\int_{B_{1/2}(x)}(\frac{1}{|z-y|})^{\frac{2\|f^+\|_{L^1(B_{1/2}(x))}}{(n-1)!|\mathbb{S}^n|}}\frac{f^+(y)}{\|f^+\|_{L^1(B_{1/2}(x))}}\ud y\\
		\leq &|x|^{-\alpha+o(1)}\int_{B_{1/2}(x)}\frac{1}{\sqrt{|z-y|}}\frac{f^+(y)}{\|f^+\|_{L^1(B_{1/2}(x))}}\ud y
	\end{align*}
	where we have used the  fact $|z-y|\leq 1$ for $z\in B_{1/4}(x)$ and $y\in B_{1/2}(x)$.
	Then by using Fubini's theorem and the above estimate, we have
	\begin{align*}
		\int^{|x|+\frac{1}{4}}_{|x|-\frac{1}{4}}e^{\mathcal{L}(f)(\gamma(t))}\ud t\leq &|x|^{-\alpha+o(1)}\int^{|x|+\frac{1}{4}}_{|x|-\frac{1}{4}}\int_{B_1(x)}\frac{1}{\sqrt{|\gamma(t)-y|}}\frac{f^+(y)}{\|f^+\|_{L^1(B_{1/2}(x))}}\ud y\ud t\\
		\leq& C|x|^{-\alpha+o(1)}\\
		=&|x|^{-\alpha+o(1)}.
	\end{align*}
	Thus, for $i\gg1$, one has
	$$\int_i^{i+1}e^{\mathcal{L}(f)(\gamma(t))}\ud t=\int_i^{i+\frac{1}{2}}e^{\mathcal{L}(f)(\gamma(t))}\ud t+\int_{i+\frac{1}{2}}^{i+1}e^{\mathcal{L}(f)(\gamma(t))}\ud t\leq i^{-\alpha+o(1)}.$$
	Hence we finish our proof.
\end{proof}	
With help of above lemma, it is easy to obtain the following corollary which has been established in \cite{CQY}.
\begin{corollary}\label{cor:normal metric}
	If the complete metric $g=e^{2u}|dx|^2$ with finite total Q-curvature is normal, then there holds
	$$\int_{\mr^n}Qe^{nu}\ud x\leq \frac{(n-1)!||\mathbb{S}^n|}{2}.$$
\end{corollary}
\begin{proof}
	Since the metric $g$ is normal, one has
$$u(x)=\mathcal{L}(Qe^{nu})+C.$$
Set the notation
\begin{equation}\label{alpha_0 def}
	\alpha_0:=\frac{2}{(n-1)!|\mathbb{S}^n|}\int_{\mr^n}Qe^{nu}\ud x.
\end{equation}
	We argue by contradiction and suppose that $\alpha_0>1$. 
	Consider the curve $\gamma(t)=tx_0$ where $x_0\in \mr^n$ satisfying $|x_0|=1$ and $t\geq 0$. Choosing the positive constant  $\epsilon=\frac{\alpha_0-1}{2}>0$, Lemma \ref{lem: line e^u} shows that there exists  an integer $s_1>0$ such that for all $s\geq s_1$, there holds
	\begin{equation}\label{t>t_1}
		\int^{s+1}_se^{u(\gamma(t))}\ud t\leq s^{-\alpha_0+\epsilon}.
	\end{equation}
	Then for the integer $k>s_1$, there holds
	\begin{align*}
		d_g(\gamma(s_1),\gamma(k))\leq &\int^{k}_{s_1}e^{u(\gamma(t))}\ud t\\
		\leq &\sum^k_{i=s_1}i^{-\alpha_0+\epsilon}\\
		=&\sum^k_{i=s_1}i^{-1-\epsilon}
	\end{align*}
which yields that
$$	d_g(\gamma(s_1),\gamma(k))<+\infty, \quad \mathrm{as}\; k\to\infty.$$
This contradicts to the assumption that the metric $g$ is complete. Thus we finish our proof.
\end{proof}
{\bf Proof of Theorem \ref{thm: tau(g)}:}

On one hand, when the metric $g$ is normal, Lemma \ref{lem: e^nLf on B_R+1 B_R-1} yields that
$$\int_{B_R(0)}e^{nu}\ud x=C\int_{B_R(0)}e^{n\mathcal{L}(Qe^{nu})}\ud x=O(R^{a_1})$$
for some constant $a_1\geq 0$ which deduces that $\tau(g)$ is finite.

On the other hand, if  the volume entropy $\tau(g)$ is finite, we follow the same argument established in the proof of Theorem 1.1 of \cite{Li 23 Q-curvature} to show that the metric $g$ is normal.

Set the function $P(x)$ defined by 
\begin{equation}\label{P def}
	P(x):=u(x)-\mathcal{L}(Qe^{nu})(x).
\end{equation}
With help of the equations \eqref{equ:conformal eqution} and \eqref{delta Lf}, one has 
\begin{equation}\label{Delta n/2P=0}
	(-\Delta)^{\frac{n}{2}}P=0.
\end{equation}
In particular, when $n$ is odd, making use of the property \eqref{Delta 1/2 circ 1/2=1} and the equation \eqref{Delta n/2P=0}, there holds
\begin{equation}\label{n+1/2}
	(-\Delta)^{\frac{n+1}{2}}P=(-\Delta)^{\frac{1}{2}}\circ(-\Delta)^{\frac{n}{2}}P=0.
\end{equation}
Thus, we can  rewrite \eqref{Delta n/2P=0} and  \eqref{n+1/2}  for all integer $n\geq 2$ as follows
\begin{equation}\label{[n+1]/2}
	(-\Delta)^{[\frac{n+1}{2}]}P=0
\end{equation}
where $[s]$ denotes the largest integer less than or equal to 
$s$.

Based on the assumption that  $\tau(g)$ is finite, then there exists a constant  $a_2\geq0$ such that,  for all $R\geq 1$,  there holds
\begin{equation}\label{e^{nu} growth}
	\int_{B_R(0)}e^{nu}\ud x\leq CR^{a_2}.
\end{equation}
Choose a small positive constant $q\in (0,1)$ such that $(\frac{a_2}{n}-1)q\leq \frac{1}{2}$.
With help of the above estimate \eqref{e^{nu} growth} and H\"older's inequality, together with Lemma \ref{lem: B_R(0)Lf},   one has 
\begin{align*}
	\int_{B_R(0)}P^+\ud x\leq &	\int_{B_R(0)}u^+\ud x+\int_{B_R(0)}|\mathcal{L}(Qe^{nu})|\ud x\\
	\leq & \frac{1}{q}\int_{B_R(0)}e^{q}\ud +O(R^n\log R)\\
	\leq &\frac{1}{q}\left(\int_{B_R(0)}e^{nu}\ud x\right)^{\frac{q}{n}}|B_R(0)|^{1-\frac{q}{n}}+O(R^n\log R)\\
	\leq &CR^{n+\frac{1}{2}}.
\end{align*}
Using  above  estimate and the equation \eqref{[n+1]/2}, Lemma \ref{lem: polynomial growth is polynomial} yields that  
$P(x)$ is a polynomial with degree satisfying
\begin{equation}\label{deg P leq n-1}
	\deg(P)\leq n-1.
\end{equation}

Now, we are going to show that the polynomial $P(x)$ must be a constant.

Firstly, we claim that 
\begin{equation}\label{P upper bound log|x|}
	P(x)\leq C\log (|x|+2).
\end{equation}
Making use of  \eqref{deg P leq n-1}, one has $$|\nabla P(x)|\leq C(1+|x|)^{n-2}.$$ Thus, for $x\in\mr^n$ with  $|x|=R\gg1$ and  any $y\in B_{|x|^{2-n}}(x)$, there holds
\begin{equation}\label{P(y)geq P(x)}
	P(y)\geq P(x)-C.
\end{equation}
With help of this  estimate \eqref{P(y)geq P(x)}, Lemma \ref{lem: B_|x|^p L(f)} and Jensen's inequality, for $|x|=R\gg1$,  one has
\begin{align*}
	\int_{B_{2R}(0)}e^{nu}\ud y\geq &\int_{B_{R^{2-n}}(x)}e^{nP(y)+n\mathcal{L}(Qe^{nu})}\ud y\\
	\geq &Ce^{nP(x)}\int_{B_{R^{2-n}}(x)}e^{n\mathcal{L}(Qe^{nu})}\ud y\\
	\geq &Ce^{nP(x)}|B_{R^{2-n}}(x)|\exp\left(\frac{1}{|B_{R^{2-n}}(x)|}\int_{B_{R^{2-n}}(x)}e^{n\mathcal{L}(Qe^{nu})}\ud y\right)\\
	\geq &Ce^{nP(x)}R^{-C}.
\end{align*}
Combing with the volume control  \eqref{e^{nu} growth}, then we obtain that 
$$P(x)\leq C\log(|x|+2).$$
Thus we prove the claim \eqref{P upper bound log|x|}.

Secondly, we show that $P(x)$ is a constant and argue by contradiction.

 If $P(x)$ is not a constant,  then the polynomial $P(x)$ must have the following decomposition
$$P(x)=H_m(x)+P_{m-1}(x)$$
where $m=\deg (P)\geq 1$,  $H_m(x)$ is a non-constant homogeneous polynomial of degree $m$ and $P_{m-1}$ is a polynomial with degree not greater than $m-1$.
Due to the estimate \eqref{P upper bound log|x|}, we must have 
\begin{equation}\label{H_m leq 0}
	H_m(x)\leq 0.
\end{equation}
Otherwise there exists $x_1\in\mr^n$ with $|x_1|=1$  such that $H_m(x_1)>0$ and for any $t>0$
$$H_m(tx_1)=t^mH_m(x_1)>0.$$
For $t\gg1$, there holds
$$P(tx_1)\geq t^mH_m(x_1)-Ct^{m-1}\geq Ct^m$$
which contradicts to \eqref{P upper bound log|x|}. With help of the control  \eqref{H_m leq 0} and the assumption  that $H_m(x)$ is not a constant, there exists $x_2\in \mr^n$ with $|x_2|=1$ such that $H_m(x_2)<0.$ Consider the ray $\gamma(t)=tx_2$ with $t>0$.  
There exists $t_1>0$ such that for $t\geq t_1$ there holds
\begin{equation}\label{P upper bound}
	P(tx_2)=H_m(tx_2)+P_{m-1}(tx_2)\leq -c_0t^m
\end{equation}
where the constant $c_0=-\frac{H_m(x_2)}{2}>0$. With help of Lemma \ref{lem: line e^u}, for any  $\epsilon_0>0$, there exists $t_2>0$ such that for  $t\geq t_2$ there holds
\begin{equation}\label{e^v leq t6-alpha-0}
	\int_{t}^{t+1}e^{\mathcal{L}(Qe^{nu})(\gamma(t))}\ud t\leq t^{-\alpha_0+\epsilon_0}
\end{equation}
where $\alpha_0$ is given by \eqref{alpha_0 def}.
Choosing an integer $T_1\geq t_1+t_2+1$ and another integer $T>T_1$, one has
\begin{align*}
	d_g(\gamma(T_1),\gamma(T))\leq &\int_{T_1}^Te^{u(\gamma(t))}\ud t\\
	\leq & \sum^{T}_{i=T_1}\int^{i+1}_{i}e^{u(\gamma(t))}\ud t\\
	=&\sum^{T}_{i=T_1}\int^{i+1}_{i}e^{P(\gamma(t))+\mathcal{L}(Qe^{nu})(\gamma(t))}\ud t\\
	\leq &\sum^{T}_{i=T_1}e^{-c_0i^m}i^{-\alpha_0+\epsilon_0}.
\end{align*}
By letting $T\to\infty$, one has
$d_g(\gamma(T_1),\gamma(T))<+\infty$
which contradicts to the assumption that the metric $g$ is complete.
Hence $P(x)$ must be a constant and then  the metric $g$ is normal.

Finally, we compute  the precise value of the finite $\tau(g)$.
On one hand, with help of Lemma \ref{lem: B_r_1|x| L(f)} and Jensen's inequality,  for any  $x_R\in \mr^n $ with $|x_R|=R\gg1$, one has
\begin{align*}
	\int_{B_{2R}(0)}e^{nu}\ud x\geq &C\int_{B_{R/2}(x_R)}e^{n\mathcal{L}(Qe^{nu})}\ud x\\
	\geq &|B_{R/2}(x_R)|\exp\left(\frac{1}{|B_{R/2}(x_R)|}\int_{B_{R/2}(x_R)}n\mathcal{L}(Qe^{nu})\ud x\right)\\
	\geq &CR^n\cdot R^{-n\alpha_0+o(1)}
\end{align*}
which yields that
\begin{equation}\label{tau(g) lower bound}
	\tau(g)\geq 1-\alpha_0.
\end{equation}
One the other hand, Lemma \ref{lem: e^nLf on B_R+1 B_R-1} yields that for any $\epsilon>0$ there exists $R_1>1$ such that for any $R>R_1$ one has
$$\int_{B_{R+1}(0)\backslash B_{R-1}(0)}e^{n\mathcal{L}(Qe^{nu})}\ud x\leq R^{n-1-n\alpha_0+\epsilon}.$$
Then for $R\gg1$, one has
\begin{align*}
	\int_{B_R(0)}e^{nu}\ud x\leq&C+C\sum^{[R]}_{i=[R_1]+1}\int_{B_{i+1}(0)\backslash B_{i-1}(0)}e^{n\mathcal{L}(Qe^{nu})}\ud x\\
			\leq &C+C\sum^{[R]}_{i=[R_1]+1}i^{n-1-n\alpha_0+\epsilon}\\
			\leq &C+C\int^{[R]+1}_{[R_1]}t^{n-1-n\alpha_0+\epsilon}\ud t.
\end{align*}
Corollary \ref{cor:normal metric} ensures that  $\alpha_0\leq 1.$ Thus we have
$$	\int_{B_R(0)}e^{nu}\ud x\leq C+\frac{C}{n-n\alpha_0+\epsilon}([R]+1)^{n-n\alpha_0+\epsilon}$$
which deduces that
$$\tau(g)\leq 1-\alpha_0+\frac{\epsilon}{n}.$$
Due to the arbitrary choice of $\epsilon$ and  the lower bound estimate \eqref{tau(g) lower bound}, one has
$$\tau(g)=1-\frac{2}{(n-1)!|\mathbb{S}^n|}\int_{\mr^n}Qe^{nu}\ud x.$$

Thus we finish our proof.

\section{Non-normal complete metrics}\label{section:non-normal metrics}

In this section, we are devoted to deal with non-normal metrics. Firstly, we show that if the volume entropy $h(g)$ is finite, then $u(x)$ has a nice decomposition.

\begin{prop}\label{lem:decomposition}
Consider the smooth  solution $u(x)$ to the equation  \eqref{equ:conformal eqution}.	If $h(g)$ is finite,  one has the following decomposition
	$$u(x)=\mathcal{L}(Qe^{nu})+P(x)$$
	where $P(x)$ is a polynomial with $\deg P\leq n-1$. 
\end{prop}
\begin{proof}
	Since $h(g)$ is finite, by choosing  $\epsilon>0$, one has 
	$$\log \int_{B_R(0)}e^{nu}\ud x\leq CR^{h(g)+\epsilon}.$$
For $R\gg1$,	with help of Jensen's inequality, there holds
	\begin{align*}
		\frac{1}{|B_R(0)|}\int_{B_R(0)} nu^+\ud x\leq &\log\left(\frac{1}{|B_R(0)|}\int_{B_R(0)} e^{nu^+}\ud x\right)\\
		\leq &\log\left(\frac{1}{|B_R(0)|}\int_{B_R(0)} (e^{nu}+1)\ud x\right)\\
		\leq &\log \left(CR^{-n}e^{CR^{h(g)+\epsilon}}+1\right)\\
		\leq &\log\left((1+C)e^{CR^{h(g)+\epsilon}}\right)\\
		\leq &CR^{h(g)+\epsilon}+\log(1+C).
	\end{align*}
Thus one has
\begin{equation}\label{u^+}
		\frac{1}{|B_R(0)|}\int_{B_R(0)}u^+\ud x\leq o(R^{h(g)+2\epsilon})
\end{equation}
Set the same notation $P(x)$ as in the proof of Theorem \ref{thm: tau(g)}
$$P(x):=u(x)-\mathcal{L}(Qe^{nu})$$
and then one has 
\begin{equation}\label{n+1/2P=01}
	(-\Delta)^{[\frac{n+1}{2}]}P=0.
\end{equation}
Making use of Lemma \ref{lem: B_R(0)Lf} and the estimate \eqref{u^+}, one has
\begin{align*}
	\frac{1}{|B_R(0)|}\int_{B_R(0)}P^+\ud x\leq &\frac{1}{|B_R(0)|}\int_{B_R(0)}u^+\ud x +\frac{1}{|B_R(0)|}\int_{B_R(0)}|\mathcal{L}(Qe^{nu})|\ud x\\
	\leq &o(R^{h(g)+2\epsilon})+O(\log R)\\
	=&o(R^{h(g)+2\epsilon})\\
	=&o(R^{[h(g)+2\epsilon]+1}).
\end{align*}
Applying the above estimate, the equation \eqref{n+1/2P=01} and  Lemma \ref{lem: polynomial growth is polynomial}, we show that 
$P(x)$ is a polynomial satisfying
$$\deg P\leq \max\{2[\frac{n+1}{2}]-2, [h(g)+2\epsilon]\}.$$
By choosing suitable small $\epsilon$, there holds
$$[h(g)+2\epsilon]\leq h(g).$$
Thus one has 
$$\deg P\leq \max\{n-1, h(g)\}.$$
If $P(x)$ is not a constant,  the polynomial $P(x)$ has the following decomposition as before
$$P(x)=H_m(x)+P_{m-1}(x)$$
where $H_m(x)$ is a homogeneous polynomial of degree $m$ and $P_{m-1}(x)$ is a polynomial of degree not greater than $m-1$.

Firstly, we claim that 
\begin{equation}\label{H_m geq 0}
	H_m(x)\geq 0.
\end{equation}
To show that $H_m(x)\geq 0$, we only need to show that $H_m(x)\geq 0$ for any $x\in\mr^n$ with $|x|=1$	since $H_m(x)$ is homogeneous. We argue by contradiction. Suppose that  there exists $x_0\in \mr^n$ with $|x_0|=1$ such that $H_m(x_0)<0.$ Following the same argument of the proof in Theorem \ref{thm: tau(g)}, we can show that along  the curvature $\gamma(t)=tx_0$, for fixed $t_1>0$ and any $t>t_1$, there holds
$$d_g(\gamma(t_1),\gamma(t))\leq C$$
which contradicts to completeness of the metric $g$. Hence we prove claim \eqref{H_m geq 0}.

Secondly, we claim that 
\begin{equation}\label{adeg P leq n-1}
	\deg P\leq n-1
\end{equation}
and  argue by contradiction.

On one hand, if $m=\deg P=n$, 
a   direct computation yields that 
$\Delta^{[\frac{n+1}{2}]}H_m$ equals to  a  constant.
Meanwhile, it is not hard  to obtain that  $\Delta^{[\frac{n+1}{2}]}P_{m-1}=0.$  Due to the equation \eqref{n+1/2P=01},  we must have
$$
	\Delta^{[\frac{n+1}{2}]}H_m=\Delta^{[\frac{n+1}{2}]}P_{m-1}=0.
$$
On the other hand, we suppose that  $m\geq n+1$.  Making use of the equation \eqref{n+1/2P=01}, there holds
\begin{equation}\label{delta H_m=-delta P_m-1}
	\Delta^{[\frac{n+1}{2}]}H_m(x)=-\Delta^{[\frac{n+1}{2}]}P_{m-1}(x).
\end{equation}
It is not hard to check that  either  $\Delta^{[\frac{n+1}{2}]}H_m(x)$ is a polynomial of  degree $m-2[\frac{n+1}{2}]$ or $\Delta^{\frac{n}{2}}H_m(x)=0$. Meanwhile, there holds that either $\Delta^{[\frac{n+1}{2}]}P_{m-1}$ is a polynomial of degree at most $m-1-2[\frac{n+1}{2}]$ or  $\Delta^{[\frac{n+1}{2}]}P_{m-1}=0.$ 
Combing  this facts with the identity \eqref{delta H_m=-delta P_m-1},  one has 
$$
	\Delta^{[\frac{n+1}{2}]}H_m(x)=\Delta^{[\frac{n+1}{2}]}P_{m-1}(x)=0.
$$
Thus, in any situation, there holds
\begin{equation}\label{n/2 Delta H_m=0}
	\Delta^{[\frac{n+1}{2}]}H_m(x)=\Delta^{[\frac{n+1}{2}]}P_{m-1}(x)=0.
\end{equation}
 However, making use of the equations  \eqref{H_m geq 0} and \eqref{n/2 Delta H_m=0},   Lemma \ref{lem: polynomial growth is polynomial} yields that $m=\deg H_m\leq n-1$ which contradicts to our assumption that $m\geq n$.

Thus we prove the claim \eqref{adeg P leq n-1} and finish our proof.

\end{proof}

Chang-Hang-Yang \cite{CHY} (See also Lemma 3.1 in \cite{MQ adv}, \cite{MQ apde}) established an interesting lemma which is crucial in the  proof of Theorem \ref{thm: h(g) leq n-2}.  For the reader's convenience, we repeat their statement as follows.

\begin{lemma}\label{CHY lemma}
	(Proposition 8.1 in \cite{CHY})
	Let $\Omega\subset \ms^n(n\geq 3)$ be an open set and $w_0\in C^\infty(\Omega)$ such that $(\Omega, e^{2w_0}g_{\ms^n})$ is complete. If the scalar curvature $R_{e^{2w_0}g_{\ms^n}}\geq-C$ for some constant $C>0$, then $w_0(x)\to\infty$ as $\mathrm{dist}_{g_{\ms^n}}(x, \partial \Omega)\to0.$
\end{lemma}
With help of above lemma, we are able to establish the following lemma.
\begin{lemma}\label{lem: u geq -log|x|}
	Consider a complete and smooth metric $g=e^{2u}|dx|^2$ on $\mr^n$.  If the scalar curvature $R_g\geq -C$ for some constant $C>0$, there holds
	$$u(x)\geq -2\log(|x|+1)-C.$$
\end{lemma}
\begin{proof}
	Via a stereographic projection 
	$$\Psi: \ms^n\backslash \{N\}\rightarrow \mr^n$$
	where $N$ is the north pole of the standard sphere $\ms^n$.  With help of such stereographic projection, the standard metric $g_{\ms^n}$ on $\ms^n$ can be written as 
	$$g_{\ms^n}=\left(\frac{2}{1+|x|^2}\right)^2|dx|^2.$$
	Thus applying Lemma \ref{CHY lemma}, one has
	$$u(x)+\log\frac{|x|^2+1}{2}\to \infty$$
	as $|x|\to\infty.$ Thus  we finish our proof.
\end{proof}

With help of above preparations, we are going to finish the proofs of  our main theorems.

{\bf Proof of  Theorem \ref{thm: h(g) leq n-2} and Theorem \ref{thm: normal decomposition}:}

	Based on our  assumption that the scalar curvature $R_g\geq -C$,  Lemma \ref{lem: u geq -log|x|} yields that
\begin{equation}\label{u^- growth}
	\int_{B_R(0)}u^-(x)\ud x=o(R^{n+1}).
\end{equation}
Set the same notation as before 
$$P(x):=u(x)-\mathcal{L}(Qe^{nu})$$
and then $P(x)$ satisfies the equation
\begin{equation}\label{n+1/2P=0}
	\Delta^{[\frac{n+1}{2}]}P=0.
\end{equation}
Making use of the estimate \eqref{u^- growth} and Lemma \ref{lem: B_R(0)Lf}, for $R\gg1$, there holds
\begin{equation}\label{P^- growth}
	\int_{B_R(0)}P^-\ud x\leq \int_{B_R(0)}u^-\ud x+\int_{B_R(0)}|\mathcal{L}(Qe^{nu})|\ud x=o(R^{n+1})
\end{equation}
Applying Lemma \ref{lem: polynomial growth is polynomial}(by setting $\varphi=-P$),   we show that $P(x)$ is a polynomial with degree  satisfying
\begin{equation}\label{deg P}
	\deg(P)\leq n-1.
\end{equation}
 With help of Lemma \ref{lem: e^nLf on B_R+1 B_R-1} and the fact $P(x)$ is a polynomial, one has
\begin{align*}
	\int_{B_R(0)}e^{nu}\ud x=&\int_{B_R(0)}e^{nP+n\mathcal{L}(Qe^{nu})}\ud x\\
	\leq &e^{CR^{\deg P}}\int_{B_R(0)}e^{n\mathcal{L}(Qe^{nu})}\ud x\\
	\leq &Ce^{CR^{\deg P}}R^C
\end{align*}
which yields that
\begin{equation}\label{h(g) finite}
	h(g)\leq \deg P. 
\end{equation}

	If $P(x)$ is not a constant, we can decompose $P(x)$ as before 
	\begin{equation}\label{decomposition P}
		P(x)=H_{m}(x)+P_{m-1}(x)
	\end{equation}
	where 
	$H_m(x)$ is a homogeneous polynomial with degree $m=\deg(P)$ and $P_{m-1}$ is a polynomial with $$\deg P_{m-1}\leq m-1.$$ 
Due  to the estimate \eqref{H_m geq 0} and the fact that $H_m$ is homogeneous, 	$m$ must  an even integer i.e.
\begin{equation}\label{deg P=even}
	\deg P \;\mathrm{is\; even}.
\end{equation}
	
Secondly, we claim that
\begin{equation}\label{claim 3}
	h(g)=\deg P.
\end{equation}

Since the non-constant homogeneous polynomial $H_m(x)\geq 0$, there exists $\delta>0$ and $x_3\in \mr^n$ with $|x_3|=1$ such that  for some constant $$H_m(y)\geq c_1>0,\quad y\in B_{\delta}(x_3).$$
Notice that  for $t>0$ there holds
$$B_1(tx_3)= t B_{\frac{1}{t}}(x_3):=\{y\in\mr^n|y=tx, x\in B_{\frac{1}{t}}(x_3)\}.$$
Immediately, for $t>\frac{1}{\delta}$, one has
$$B_1(tx_3)\subset tB_{\delta}(x_3).$$
Then,  for $t>\frac{1}{\delta}$ and $x\in B_1(tx_3)$, we have 
$$P(x)=t^mH_m(\frac{x}{t})+P_{m-1}(x)\geq c_1t^m-C(t+1)^{m-1}.$$
Thus for $R\gg1$ and $x_R:=Rx_3$, one has
\begin{equation}\label{P geq R^m}
	P(x)\geq c_2R^m
\end{equation}
for some positive constant $c_2$.
Making use of Lemma \ref{lem: B_1 L(f)}, the estimate \eqref{P geq R^m} and Jensen's inequality, for $R\gg1$, one has  
\begin{align*}
	\int_{B_{2R}(0)}e^{nu}\ud x \geq & \int_{B_1(x_R)}e^{nP+n\mathcal{L}(Qe^{nu})}\ud x\\
		\geq &e^{nc_2R^m}\int_{B_1(x_R)}e^{n\mathcal{L}(Qe^{nu})}\ud x\\
		\geq &e^{nc_2R^m}|B_1(x_R)|\exp\left(\frac{1}{|B_1(x_R)|}\int_{B_1(x_R)}n\mathcal{L}(Qe^{nu})\ud x\right)\\
		\geq &Ce^{nc_2R^m-(n\alpha_0+o(1))\log R}.
\end{align*}
Then we obtain that 
$$\log V_g(B_{2R}(0))\geq nc_2R^m-(n\alpha_0+o(1))\log R-C$$
which yields that the volume entropy
$$h(g)\geq m.$$
Combing with the upper bound estimate \eqref{h(g) finite}, we prove the claim \eqref{claim 3}.

Thus, making use of \eqref{deg P}, \eqref{deg P=even} and \eqref{claim 3}, we prove Theorem \ref{thm: h(g) leq n-2}.

With help of \eqref{deg P}, for $|x|\gg1$ and  any $y\in B_{|x|^{2-n}}(x)$, one has
\begin{equation}\label{P(x) geq P(y)-C }
	P(x)\geq P(y)-C.
\end{equation}
Applying the estimate \eqref{P(x) geq P(y)-C }, Lemma \ref{lem: B_|x|^p L(f)} and  Lemma \ref{lem: u geq -log|x|}, for $|x|\gg1$,  there holds
\begin{align*}
	&P(x)\\
	\geq & -C+\frac{1}{|B_{|x|^{2-n}}(x)|}\int_{B_{|x|^{2-n}}(x)}P(y)\ud y\\
	\geq &-C+\frac{1}{|B_{|x|^{2-n}}(x)|}\int_{B_{|x|^{2-n}}(x)}u(y)\ud y-\frac{1}{|B_{|x|^{2-n}}(x)|}\int_{B_{|x|^{2-n}}(x)}\mathcal{L}(Qe^{nu})\ud y\\
	\geq &-C-2\log(|x|+1)-(-\alpha_0+o(1))\log|x|\\
	\geq &-C\log(|x|+2).
\end{align*}
Thus we finish the proof of Theorem \ref{thm: normal decomposition}.

\section{Control of $\tau(g)$ by scalar curvature}\label{sec: finite tau(g)}

As noted earlier, Chang, Qing, and Yang demonstrated in \cite{CQY} that a metric with finite total Q-curvature is normal if the scalar curvature remains non-negative near infinity in the four-dimensional case. For higher dimensions, a similar proof applies (see \cite{Fa}, \cite{NX}, \cite{Li 23 Q-curvature}).
\begin{prop}\label{prop}
		Consider a  smooth and complete metric $g=e^{2u}|dx|^2$ on $\mr^n$  with finite total Q-curvature where the integer $n\geq 3$. Suppose that the scalar curvature $R_g\geq 0$ near infinity. Then the metric $g$ is normal.
\end{prop}
\begin{proof}
	Set the function $h:=u-\mathcal{L}(Qe^{nu})$. Due to the equation \eqref{[n+1]/2}, there holds
	$$\Delta^{[\frac{n+1}{2}]}h=0.$$
Then	it is not hard to 	imitate the proof of Theorem 4.2 in \cite{Li 23 Q-curvature} to obtain our result and we omit the details.
\end{proof}

To complete the proofs of Theorem \ref{thm: tua(g) leq 1} and Theorem \ref{thm: tau(g) =0}, we need  the crucial lemmas established in \cite{Li 24 revisited}. Recall  $\mathcal{L}(Qe^{nu})$ given by \eqref{mathcal L} and $\alpha_0$  given by \eqref{alpha_0 def}. We repeat the statements of lemmas as follows.
\begin{lemma}\label{lem: 4.1}(Lemma 2.5 in \cite{Li 24 revisited})
	As $r\to\infty$, there holds
	$$r^2\cdot\frac{1}{|\partial B_r(0)|}\int_{\partial B_r(0)}(-\Delta)\mathcal{L}(Qe^{nu})\ud \sigma\to (n-2)\alpha_0.$$
\end{lemma}
\begin{lemma}\label{lem: 4.2}(Lemma 2.7 in \cite{Li 24 revisited})
	As $r\to\infty$, there holds
	$$r^2\cdot\frac{1}{|\partial B_r(0)|}\int_{\partial B_r(0)}|\nabla \mathcal{L}(Qe^{nu})|^2\ud \sigma\to \alpha_0^2.$$
\end{lemma}
\begin{lemma}\label{lem: 4.3}(Lemma 2.8 in \cite{Li 24 revisited})
	There holds
	$$\frac{1}{|\partial B_r(0)|}\int_{\partial B_r(0)}\mathcal{L}(Qe^{nu})\ud\sigma =(-\alpha_0+o(1))\log r$$
	where $o(1)\to 0$ as $r\to\infty.$
\end{lemma}

{\bf Proof of Theorem \ref{thm: tua(g) leq 1} and Theorem \ref{thm: tau(g) =0}:}

Firstly, Proposition \ref{prop} demonstrates that, based on the sign control of the scalar curvature at infinity, the metric 
$g$  is normal, meaning that
\begin{equation}\label{u is normal}
	u(x) = \mathcal{L}(Qe^{nu}) + C.
\end{equation}
A direct computation reveals that the scalar curvature can be represented as follows:
\begin{equation}\label{scalar curvature}
	R_g=2(n-1)e^{-2u}(-\Delta u-\frac{n-2}{2}|\nabla u|^2).
\end{equation}

If $R_g\geq 0$ near infinity, for $r\gg1$, the equation \eqref{scalar curvature} ensures that 
$$r^2\cdot\frac{1}{|\partial B_r(0)|}\int_{\partial B_r(0)}\left(-\Delta u-\frac{n-2}{2}|\nabla u|^2\right)\ud \sigma\geq 0.
$$
Combining \eqref{u is normal} with Lemma \ref{lem: 4.1} and Lemma \ref{lem: 4.2}, and letting $r\to\infty$, we obtain that 
\begin{equation}\label{ R_g geq 0 range}
	(n-2)\alpha_0-\frac{n-2}{2}\alpha_0^2\geq 0.
\end{equation}
Since the metric $g$ is complete, combing Corollary \ref{cor:normal metric} with  \eqref{ R_g geq 0 range}, one has
\begin{equation}\label{0leq alpha_0 leq 1}
	0\leq \alpha_0\leq 1.
\end{equation}
Applying Theorem \ref{thm: tau(g)} and the estimate \eqref{0leq alpha_0 leq 1}, we obtain the volume control
$$\tau(g)\leq 1.$$

If $R_g\geq C>0$ near infinity, for $r\gg1$, we make use of Jensen's inequality and Lemma \ref{lem: 4.3} to obtain that 
\begin{align*}
	&r^2\cdot\frac{1}{|\partial B_r(0)|}\int_{\partial B_r(0)}\left(-\Delta u-\frac{n-2}{2}|\nabla u|^2\right)\ud \sigma\\
	=&r^2\cdot\frac{1}{|\partial B_r(0)|}\int_{\partial B_r(0)}\frac{1}{2(n-1)}R_ge^{2u}\ud\sigma\\
	\geq &Cr^2\cdot\frac{1}{|\partial B_r(0)|}\int_{\partial B_r(0)}e^{2u}\ud\sigma\\
	\geq &Cr^2\exp\left(\frac{1}{|\partial B_r(0)|}\int_{\partial B_r(0)}2u\ud\sigma\right)\\
	=&Cr^2\cdot r^{-2\alpha_0+o(1)}.
\end{align*}
Combing Lemma \ref{lem: 4.1} with Lemma \ref{lem: 4.2},   the left side of above estimate is bounded from above as $r\to\infty$. Using this fact along with the representation on the right side of the above estimate, we obtain that
$$\alpha_0\geq 1.$$
Combing with Corollary \ref{cor:normal metric}, we have
\begin{equation}\label{alpha_0=1}
	\alpha_0=1.
\end{equation}
 Making use of Theorem \ref{thm: tau(g)} and  the equation \eqref{alpha_0=1}, we finally obtain that 
 $$\tau(g)=0.$$

Thus we finish our proofs.

\end{document}